\documentclass{article}
\usepackage{amsmath,amsthm,amsopn,amssymb}
\allowdisplaybreaks

\numberwithin{equation}{section}
\newtheorem{theorem}{Theorem}[section]
\newtheorem{lemma}{Lemma}[section]
\newtheorem{example}{Example}[section]

\makeatother

\begin{document}

\title{The Newton-Shamanskii method for solving a quadratic matrix equation arising in quasi-birth-death problems}

\author{Pei-Chang Guo \thanks{ e-mail: guopeichang@pku.edu.cn}\\
School of Mathematical Sciences, Peking University, Beijing 100871, Beijing, China}

\date{}

\maketitle
\begin{abstract}
In order to determine the stationary distribution for discrete time
quasi-birth-death Markov chains, it is necessary to find the minimal nonnegative solution of a quadratic matrix equation. We apply the Newton-Shamanskii method for solving the equation. We show that the sequence of matrices generated by the Newton-Shamanskii method is monotonically increasing and converges to the minimal nonnegative solution of the equation. Numerical experiments show the effectiveness of our method.

\vspace{2mm} \noindent \textbf{Keywords}: quadratic matrix equation, quasi-birth-death problems, Newton-Shamanskii method, minimal nonnegative solution.
\end{abstract}

\section{Introduction}

We first introduce some necessary notation for the paper. For any matrices $B=[b_{ij}]\in \mathbb{R}^{m\times n}$, we write $B\geq 0 (B>0)$ if $b_{ij} \geq 0 (b_{ij}> 0)$ holds for all $i,j$. For any matrices $A,B \in \mathbb{R}^{n\times n}$, we write $A\geq B (A > B)$ if $a_{ij}\geq b_{ij}(a_{ij} > b_{ij})$ for all $i, j$. For any vectors $x,y \in \mathbb{R}^n$ ,we write $x\geq y (x>y)$ if $x_i\geq y_i (x_i>y_i)$ holds for all $i=1, \cdots,n$. The vector of all ones is denoted by e, i.e., $e=(1, 1, \cdots, 1)^T$. The identity matrix is denoted by $I$.

In this paper, we consider the quadratic matrix equation (QME)
\begin{equation}\label{qme}
    \mathcal{Q}(X)=AX^2+BX+C=0,
\end{equation}
where $A, B, C, X \in \mathbb{R}^{n\times n}$, $A, B+I, C \geq 0$, $A+B+I+C$ is irreducible and $(A+B+C)e=e$.

The quadratic matrix equation (\ref{qme}) has applications in quasi-birth-death processes (QBD)\cite{qbini}. The elementwise minimal nonnegative solution $S$ of the equation (\ref{qme}) is of interest in the applications. The rate $\rho$ of a QBD Markov chain is defined by
\begin{equation}\label{qrho}
    \rho= p^T(B+I+2A)e
\end{equation}
where $p$ is the stationary probability
vector of stochastic matrix $A+B+I+C$, i.e., $p^T(A+B+I+C)=p^T$ and $p^Te=1$. We refer the readers to the monograph \cite{qbini} for the details. We recall that a QBD is positive recurrent if $\rho <1$, null recurrent if $\rho =1$ and transient if $\rho > 1$. Throughout this paper, we always assume that the QBD is positive recurrent.

 There have been several numerical methods for solving the equation.
 Some linearly convergent fixed point iterations are analyzed in \cite{qlinear}
and the references therein. In \cite{qnewton} Latouche showed that Newton's algorithm for this equation is well defined and the matrix sequences is  monotonically increasing and quadratically convergent. The invariant subspaces method approximates the minimal nonnegative solution $S$ quadratically by approximating the left invariant subspace of a suitable block companion matrix \cite{qbini,crqbd1}. Latouche and Ramaswami propose a logarithmic reduction algorithm based on a divide-and conquer strategy in \cite{qlrqbd}, producing sequences of approximations which converge quadratically to $S$.
Bini and Meini et al. devise a quadratically convergent and
numerically stable algorithm \cite{crqbd1,crqbd2,crqbd3,crqbd4,qbd} for the computation of $S$ based on a functional representation
of cyclic reduction, which applies to general M/G/1 type Markov chains \cite{stoch}
and which extends the method of Latouche and Ramaswami \cite{qlrqbd}.
%

In this paper, we apply the Newton-Shamanskii method to equation \eqref{qme}. We show that, starting with a suitable initial guess, the sequence of the iterative matrices generated by the Newton-Shamanskii method is monotonically increasing and converges to the minimal nonnegative solution of QME \eqref{qme}. Numerical experiments show that the Newton-Shamanskii method is effective and outperforms the Newton method.

The rest of the paper is organized as follows. In section 2 we recall Newton's method and present the Newton-Shamanskii iterative procedure. Insection3 we prove the monotone convergence result for the Newton-Shamanskii method. In section 4 we present some
numerical examples, which show that our new algorithm is faster
than Newton method. In section 5, we give our conclusions.

\section{Newton-Shamanskii Method}

The function $\mathcal{Q}$ in \eqref{qme} is a mapping from $\mathbb{R}^{n\times n}$ into
itself and the Fr\'{e}chet derivative of $\mathcal{Q}$ at $x$ is a linear map $\mathcal{Q}^{'}_X:  \mathbb{R}^{n\times n}\rightarrow \mathbb{R}^{n\times n}$ given by
\begin{equation}\label{qfdao}
   \mathcal{Q}^{'}_X(Z)= AZX+AXZ+BZ.
\end{equation}
The second derivative at $X$, $\mathcal{Q}^{''}_X : \mathbb{R}^{n\times n}\rightarrow \mathbb{R}^{n\times n}$, is given by
\begin{equation}\label{q2fdao}
    \mathcal{Q}^{''}_X(Z_1,Z_2)=AZ_1Z_2+AZ_2Z_1.
\end{equation}

For a given $X_0$, the Newton sequence for the solution of $\mathcal{Q}(X) = 0$ is
\begin{equation}\label{qnewit}
 X_{k+1}=X_k-(\mathcal{Q}^{'}_{X_k})^{-1}\mathcal{Q}({X_k}) ,\quad k=0,1 ,2, \cdots,
\end{equation}
 provided that $\mathcal{Q}^{'}_{X_k}$ is invertible for all $k$. By equation (\ref{qfdao}), the Newton iteration (\ref{qnewit}) is equivalent to
\begin{eqnarray}\label{qnewz}
                  \left\{  \begin{array}{c}
                 AZX _k+(AX_k+B)Z=-Q(X_k),\\
                  X_{k+1}=X_k+Z , \quad k=0,1 ,2, \cdots
                  \end{array}\right.
  \end{eqnarray}
    or
    \begin{equation}\label{qnewx}
        AX_{k+1}X_k+(AX_k+B)X_{k+1}=AX_k^2-C, \quad k=0,1 ,2, \cdots.
    \end{equation}

As we see in \cite{qnewton}, for the nonlinear equation $\mathcal{Q}(X)=0$ with the minimal nonnegative solution $S$, the sequence generated by Newton method will converge quadratically and globally to the the solution $S$. However, there is a disadvantage with Newton method. At every iteration step, a Sylvester-type equation
\begin{equation*}
     A_1XB_1^T+A_2XB_2^T=E.
\end{equation*}
is needed to solve. The Bartels-Stewart method and the Hessenberg-Schur method can be employed to solve the Sylvester-type equation \cite{laub}, where the QZ algorithm is involved. When solving the Sylvester-type equation, a transformation method is used which employs the QZ algorithm to structure the equation in such a way that it can be solved columnwise by a back substitution technique. The work count of the floating point operations involved in QZ algorithm is large compared with the back substitution \cite{laub}. In order to save the overall cost, we would like to reuse the special coefficient matrix structure form produced by QZ algorithm. We present the Newton-Shamanskii algorithm for QME(\ref{qme}) as follows.

\vspace{3mm}

\textbf{Newton-Shamanskii algorithm for QME(\ref{qme})}\\

Given initial value $X_{0}$, for $k=0,1,\cdots$
\begin{eqnarray}
\label{qlikea}   X_{k,0}&=&X_k-(\mathcal{Q}^{'}_{X_k})^{-1}\mathcal{Q}({X_k}) ,\\
\label{qlikeb}  X_{k,s} &=& X_{k,s-1}-(\mathcal{Q}^{'}_{X_k})^{-1}\mathcal{Q}({X_{k,s-1}}) , \quad s=1,2,\cdots , n_k, \\
 \label{qlikec} X_{k+1} &=& X_{k,n_k}
\end{eqnarray}

\section{Convergence Analysis}
In this section, we prove a monotone convergence result for Newton-Shamanskii method for QME \eqref{qme}.
\subsection{preliminary}
We first recall that a real square matrix $A$ is called a Z-matrix if
all its off-diagonal elements are nonpositive. Note that any Z-matrix A can be
written as $sI-B$ with $B \geq 0$. A Z-matrix $A$ is called an M-matrix if $s\geq \rho(B)$,
where $\rho(\cdot)$ is the spectral radius; it is a singular M-matrix if $s=\rho(B)$ and a
nonsingular M-matrix if $s>\rho(B)$.
We will make use of the following result (see \cite{varga}).
\begin{lemma}\label{qyubei1}
For a Z-matrix $A$, the following are equivalent:
\begin{itemize}
  \item [$(a)$] $A$ is a nonsingular M-matrix.
  \item [$(b)$] $A^{-1}\geq 0$ .
  \item [$(c)$] $Av>0$ for some vector $v>0$.
  \item [$(d)$] All eigenvalues of $A$ have positive real parts.
\end{itemize}
\end{lemma}
The next result is also well known and also can be found in \cite{varga}.
\begin{lemma}\label{qyubei2}
Let $A$ be a nonsingular M-matrix. If $B \geq A$ is a Z-matrix, then $B$ is also  nonsingular M-matrix . Moreover, $B^{-1}\leq A^{-1}$.
\end{lemma}
We recall the property of the minimal nonnegative solution $S$ for QME \eqref{qme}, see \cite{qnewton,qbini} for more details.
\begin{theorem}\label{qthm0}
If the quasi-birth-death process is positive recurrent, i.e., rate $\rho$ defined by \eqref{qrho} satisfies that $\rho <1$, then the matrix
\begin{equation*}
    -[(S^T\otimes A+I\otimes AS)+I\otimes B]
\end{equation*}
is a nonsingular M-matrix.
\end{theorem}
\subsection{Monotone convergence}
The next lemma displays the monotone convergence properties of Newton iteration for QME \eqref{qme}.  \begin{lemma}\label{qlemm1}
Suppose that a matrix $X$ is such that
\begin{itemize}
  \item [(i)] $\mathcal{Q}(X)\geq 0$,
  \item [(ii)] $0\leq X\leq S$,
  \item [(iii)]$-[(X^T\otimes A+I\otimes AX)+I\otimes B]$ is a nonsingular M-matrix.
\end{itemize}
Then there exists the matrix
\begin{equation}\label{qzheng1}
    Y=X-(\mathcal{Q}^{'}_{X})^{-1}\mathcal{Q}(X)
\end{equation}
such that
\begin{itemize}
  \item [(a)] $\mathcal{Q}(Y)\geq 0$,
  \item [(b)] $0\leq X \leq Y\leq S$,
  \item [(c)]$-[(Y^T\otimes A+I\otimes AY)+I\otimes B]$ is a nonsingular M-matrix.
\end{itemize}
\end{lemma}
\begin{proof}
$\mathcal{Q}^{'}_{X}$ is invertible and the matrix $Y$ is well defined by (iii) and Lemma \ref{qyubei1}.
Because $\mathcal{Q}(X)\geq 0$ and $-[(X^T\otimes A+I\otimes AX)+I\otimes B]^{-1}\geq0$ by (iii) and Lemma \ref{qyubei1}, we have $vec(Y)\geq vec(X)$ and thus $Y\geq X$.
From equation (\ref{qzheng1}) and Taylor formula,
we have
   \begin{eqnarray*} \mathcal{Q}(Y)&=&\mathcal{Q}(X)+\mathcal{Q}^{'}_{X}(Y-X)+\frac{1}{2}\mathcal{Q}^{''}_X(Y-X,Y-X)\\
       \nonumber &=&\frac{1}{2}\mathcal{Q}^{''}_X(Y-X,Y-X)  \\
      \nonumber &=& A(Y-X)^2 \geq 0
     \end{eqnarray*}
We now prove (b). From the equivalent form of \eqref{qzheng1}
\begin{equation*}
   \nonumber  A Y X+(AX+B)Y = AX^2-C
   \end{equation*}
   and the equation
   \begin{equation*}
    AS^2+BS+C = 0,
   \end{equation*}
 we get
 \begin{eqnarray*}
   \nonumber   A(Y-S)X+(AX+B)(Y-S)&=& AX^2-C-ASX-AXS-BS \\
      \nonumber &=& A(X-S)(X-S)\\
      \nonumber &\geq & 0.
     \end{eqnarray*}

Note that $-[(X^T\otimes A+I\otimes AX)+I\otimes B]$ is a nonsingular M-matrix, therefore by Lemma \ref{qyubei1} we get $vec(S-Y)\geq 0$, which is $S-Y\geq 0$. Note that $Y\geq X$, we have proved (b).

From $0\leq Y \leq S$, we know $$-[(Y^T\otimes A+I\otimes AY)+I\otimes B] \geq -[(S^T\otimes A+I\otimes AS)+I\otimes B],$$ and we know $-[(S^T\otimes A+I\otimes AS)+I\otimes B]$ is a nonsingular M-matrix, so
 $-[(Y^T\otimes A+I\otimes AY)+I\otimes B]$ is a nonsingular M-matrix by Lemma \ref{qyubei2}.
\end{proof}
The next lemma is an extension of Lemma \ref{qlemm1}, which will be the theoretical basis of monotone convergence result of Newton-Shamanskii method for QME \eqref{qme}.

\begin{lemma}\label{qlemm2}
Suppose that a matrix $X$ is such that
\begin{itemize}
  \item [(i)] $\mathcal{Q}(X)\geq 0$,
  \item [(ii)] $0\leq X\leq S$,
  \item [(iii)]$-[(X^T\otimes A+I\otimes AX)+I\otimes B]$ is a nonsingular M-matrix.
\end{itemize}
Then for any matrix $N$ with $0\leq N\leq X$, there exists the matrix
\begin{equation}\label{qzheng2}
    Y=X-(\mathcal{Q}^{'}_{N})^{-1}\mathcal{Q}(X)
\end{equation}
such that
\begin{itemize}
  \item [(a)] $\mathcal{Q}(Y)\geq 0$,
  \item [(b)] $0\leq X\leq Y\leq S$,
  \item [(c)]$-[(Y^T\otimes A+I\otimes AY)+I\otimes B]$ is a nonsingular M-matrix.
\end{itemize}
\end{lemma}
\begin{proof}
First, because $0\leq N\leq X$, we get $$-[(N^T\otimes A+I\otimes AN)+I\otimes B]\geq -[(X^T\otimes A+I\otimes AX)+I\otimes B]$$. From (iii) and Lemma \ref{qyubei2} we know $\mathcal{Q}^{'}_{N}$ is invertible and the matrix $Y$ is well defined such that $0\leq X \leq Y$. Let
\begin{equation*}
    \hat{Y}=X-(\mathcal{Q}^{'}_{X})^{-1}\mathcal{Q}(X),
\end{equation*}
we have $\hat{Y} \geq Y$ from Lemma \ref{qyubei2}. Note that $\hat{Y}\leq S$ by Lemma \ref{qlemm1}, so we have proved (b) $0\leq Y\leq S$. Note that $$-[(\hat{Y}^T\otimes A+I\otimes A\hat{Y})+I\otimes B]$$is a nonsingular M-matrix by Lemma \ref{qlemm1} and $\hat{Y} \geq Y$, we have $-[(Y^T\otimes A+I\otimes AY)+I\otimes B]$ is a nonsingular M-matrix by Lemma \ref{qyubei2}. Last, from Taylor formula and $\mathcal{Q}^{'}_{N}(Y-X)+\mathcal{Q}(X)=0$, we have
   \begin{eqnarray*}
   \nonumber \mathcal{Q}(Y)&=&\mathcal{Q}(X)+\mathcal{Q}^{'}_{X}(Y-X)+\frac{1}{2}\mathcal{Q}^{''}_X(Y-X,Y-X)\\
 \nonumber &=&\mathcal{Q}(X)+\mathcal{Q}^{'}_{N}(Y-X)+(\mathcal{Q}^{'}_{X}-\mathcal{Q}^{'}_{N})(Y-X)+\frac{1}{2}\mathcal{Q}^{''}_X(Y-X,Y-X)\\
\nonumber &=&(\mathcal{Q}^{'}_{X}-\mathcal{Q}^{'}_{N})(Y-X)+\frac{1}{2}\mathcal{Q}^{''}_X(Y-X,Y-X)\\
\nonumber &=&\mathcal{Q}^{''}_{X}(X-N,Y-X)+\frac{1}{2}\mathcal{Q}^{''}_X(Y-X,Y-X)\\
      \nonumber &=& A(X-N)(Y-X)+A(Y-X)(X-N)+A(Y-X)^2 \\
      \nonumber &\geq &0.
     \end{eqnarray*}
\end{proof}
Using Lemma \ref{qlemm2}, we can arrive at the following monotone convergence result of Newton-Shamanskii method for QME \eqref{qme}.
\begin{theorem}\label{qthm1}
Suppose that a matrix $X_0$ is such that
\begin{itemize}
  \item [(i)] $\mathcal{Q}(X_0)\geq 0$,
  \item [(ii)] $0\leq X_0\leq S$,
  \item [(iii)]$-[(X_0^T\otimes A+I\otimes AX_0)+I\otimes B]$ is a nonsingular M-matrix.
\end{itemize}
Then the Newton-Shamanskii algorithm (\ref{qlikea}) (\ref{qlikeb}) (\ref{qlikec}) generates a sequence $\{X_k\}$
such that $X_k \leq X_{k+1}\leq S$ for all $k \geq 0$, and $ \lim_{k \to \infty} X_k=S$.
\end{theorem}
\begin{proof}
We prove the theorem by mathematical induction. From Lemma \ref{qlemm2}, we have
\begin{equation*}
    X_0\leq X_{0,0}\leq \cdots \leq X_{0,n_0}=X_1 \leq S,
\end{equation*}
\begin{equation*}
    \mathcal{Q}(X_1)\geq 0,
\end{equation*}
and know that
\begin{equation*}
    -[(X_1^T\otimes A+I\otimes AX_1)+I\otimes B]
\end{equation*}
 is a nonsingular M-matrix.
Assume
\begin{equation*}
    \mathcal{Q}(X_i)\geq 0,
\end{equation*}
\begin{equation*}
    X_0\leq X_{0,0}\leq \cdots \leq X_{0,n_0}=X_1 \leq \cdots \leq X_{i-1,n_{i-1}}=X_{i} \leq S,
\end{equation*}
and
$-[(X_i^T\otimes A+I\otimes AX_i)+I\otimes B]$ is a nonsingular M-matrix.
Again by Lemma \ref{qlemm2} we have
\begin{equation*}
    \mathcal{Q}(X_{i+1})\geq 0,
\end{equation*}
\begin{equation*}
    X_i \leq X_{i,0}\leq \cdots \leq X_{i,n_i}=X_{i+1}\leq S,
\end{equation*}
and $-[(X_{+1}i^T\otimes A+I\otimes AX_{i+1})+I\otimes B]$ is a nonsingular M-matrix.
Therefore we have proved inductively the sequence $\{X_k\}$ is monotonically increasing and bounded above by $S$. So it has a limit $X_*$ such that $X_* \leq S$.
Let $i\rightarrow \infty$ in $X_{i+1}\geq X_{i,0}=X_i-(\mathcal{Q}^{'}_{X_i})^{-1}\mathcal{Q}({X_i})\geq 0$, we see that $\mathcal{Q}({X_*})=0$. Since $X_*\leq S$, and $S$ is the minimal nonnegative solution of QME (\ref{qme}), we get $X_*=S$.
\end{proof}

\section{Numerical Experiments}
We remark that the Newton-Shamanskii method differs from Newton's method in that the  Fr\'{e}chet derivative is not updated at every iteration step. That is to say the coefficient matrix pairs of the Sylvester-type equation are evaluated and reduced with QZ algorithm after several inner iteration steps. So, while more iterations will be needed than for Newton's method, the overall cost of the Newton-Shamanskii method will be less. Our numerical experiments confirm the efficiency of the Newton-Shamskii method for QME (\ref{qme}).

The numerical tests were performed on a laptop (2.4 Ghz and 2G Memory) with MATLAB R2013a. We use $X_0=0$ as the initial iteration value of the Newton-like method. As is reported in \cite{laub}, the Hesseberg-Schur method is faster than the Bartels-Stewart method when solving the general Sylvester-type equation
\begin{equation*}
    A_1XB_1^T+A_2XB_2^T=E.
\end{equation*}
So in Newton iteration we adopt the Hesseberg-Schur method for solving the Sylvester-type equation. In Newton-Shamanskii iteration, we can reuse the reduced coefficient matrix in the back substitution step when solving Sylvester-type equation, so we adopt the Bartels-Stewart method to solve the Sylvester-type equation. That is to say in the first call to QZ algorithm, we reduce $A_1$ to quasi-upper-triangular form.

About how to choose the optimal scalars $n_i$ in the Newton-like  algorithm \eqref{qlikeb},  we have no theoretical results. In our extensive numerical experiments, we update the Fr\'{e}chet derivative every $m=2$ steps.

We define the number of the evaluation of the Fr\'{e}chet derivative in the algorithm as the outer iteration steps, which is $k+1$ for an approximate solution $x_{k,l}$ in the Newton-Shamanskii algorithm.

The outer iteration steps (denoted as ``it"),  the
elapsed CPU time in seconds (denoted as ``time"), and the normalized residual
(denoted as ``NRes" ) are used to measure the
feasibility and effectiveness of our new method, where ``NRes" is
defined as
\begin{equation*}
\mbox{NRes}=\frac{\parallel A\tilde{X}^2+B\tilde{X}+C\parallel}{\parallel \tilde{X}\parallel(\parallel A\parallel \parallel \tilde{X}\parallel + \parallel B\parallel)+\parallel C\parallel},
\end{equation*}
where $\parallel\cdot\parallel$ denotes the infinity-norm of the matrix and $\tilde{X}$ is an approximate solution to the minimal nonnegative solution of (\ref{qme}).
Numerical experiments show that the Newton-Shamanskii method are more efficient than Newton method .
\begin{example} We use the example in \cite{qlrqbd,qbd} to test our algorithm. In this example we construct a quasi-birth-death problem defined by the $n \times n$
matrices $A=W$, $B=W-I$, $C =W+\delta I$, where $I$ is the identity matrix, $W$ is a matrix having null diagonal entries
and constant off-diagonal entries,and $0 < \delta< 1$ . As was observed in \cite{qlrqbd}, the rate
$\rho = p^T (B+I+2A)e$, where $p^T (A + B+I + C) = p^T$ and $p^T e = 1$, is exactly $1-\delta$. We
have tested with three different $\delta$ values and three problem sizes. Tables \ref{tab1}, Table \ref{tab2} and Table \ref{tab3}, report the results obtained with sizes $n = 20$, $n=100$ and $n = 200$, respectively.
\end{example}

\begin{table}
\begin{center}
\caption{Comparison of the numerical results when $n=20$}
\begin{tabular}{|c|c|c|c|c|} \hline

  $\delta$ & Method  & time & it & NRes \\
  5.0e-1 & Newton  & 0.013&  5 & 4.77e-16 \\
  5.0e-1 & Newton-Shamanskii & 0.009 & 3 & 2.38e-14\\
  1.0e-1 & Newton & 0.036 & 7 & 1.61e-16 \\
  1.0e-1 & Newton-Shamanskii & 0.012 & 5 & 9.25e-16 \\
  1.0e-3 & Newton& 0.043 & 13 & 8.70e-16 \\
   1.0e-3 & Newton-Shamanskii & 0.036 & 9 & 3.00e-16 \\
  \hline
\end{tabular}\label{tab1}
\end{center}
\end{table}
\begin{table}
\begin{center}
\caption{Comparison of the numerical results when $n=100$}
\begin{tabular}{|c|c|c|c|c|}\hline
  $\delta$ & Method  & time & it & NRes \\
  5.0e-1 & Newton  & 0.142&  5 & 1.24e-15 \\
  5.0e-1 & Newton-Shamanskii & 0.110 & 3 & 2.50e-14\\
  1.0e-1 & Newton & 0.190 & 7 & 1.21e-15 \\
  1.0e-1 & Newton-Shamanskii & 0.168 & 5 & 1.60e-15 \\
  1.0e-3 & Newton& 0.444 & 13 & 1.60e-15\\
   1.0e-3 & Newton-Shamanskii & 0.359 & 9 & 6.14e-16 \\
  \hline
\end{tabular}\label{tab2}
\end{center}
\end{table}
\begin{table}
\begin{center}
\caption{Comparison of the numerical results when $n=200$}
\begin{tabular}{|c|c|c|c|c|}\hline
  $\delta$ & Method  & time & it & NRes \\
  5.0e-1 & Newton  & 1.026&  5 & 9.40e-15 \\
  5.0e-1 & Newton-Shamanskii & 0.746 & 3 & 2.34e-14\\
  1.0e-1 & Newton & 1.433& 7 & 2.18e-15 \\
  1.0e-1 & Newton-Shamanskii & 1.200 & 5 & 1.25e-15 \\
  1.0e-3 & Newton& 4.798& 13 & 5.64e-15\\
   1.0e-3 & Newton-Shamanskii & 4.271 & 9 & 2.50e-15\\
  \hline
\end{tabular}\label{tab3}
\end{center}
\end{table}

\section{Conclusions}
In this paper, we apply the Newton-Shamanskii method to the quadratic matrix equation arising from the analysis of quasi-birth-death processes. The convergence analysis shows that this method is feasible and the minimal nonnegative solution of the quadratic matrix equation can be obtained. Numerical experiments show that the Newton-Shamanskii method outperforms Newton method.

\end{document}